\documentclass[a4paper,10pt]{amsart}
\usepackage{graphicx}
\usepackage{amssymb,amsmath,amstext,amscd}

\theoremstyle{plain}
\newtheorem{thm}{Theorem}[section]

\newtheorem{lemma}{Lemma}[section]
\newtheorem*{corollary}{Corollary}
\newtheorem*{main theorem}{Theorem}

\theoremstyle{definition}

\begin{document}

\title{A note on conformal equivalences}
\author{Do-Hyung Kim}
\address{Department of Mathematics, College of Natural Science, Dankook University,
San 29, Anseo-dong, Dongnam-gu, Cheonan-si, Chungnam, 330-714,
Republic of Korea} \email{mathph@dankook.ac.kr}


\begin{abstract}
A new characterization of conformal transformations is given. By
use of this, the general form of conformal transformation on
two-dimensional Minkowski space is given and its conformal
structure is analyzed.
\end{abstract}

\maketitle

\section{Introduction} \label{section:1}

In conformal geometry, Liouville's theorem states that the only
conformal maps in semi-Euclidean space of dimension bigger than
two are those generated by isometries, homotheties and inversions.
Therefore, if we want to find conformal maps which are bijective
on $n$-dimensional semi-Euclidean space $\mathbb{R}^n_\nu$, then
the maps are composite of homotheties, isometries, since inversion
can not be defined on the whole of $\mathbb{R}^n_\nu$. However,
this is not the case for $n=2$.

For two-dimensional cases, the corresponding spaces are Euclidean
plane $\mathbb{R}^2$ and two-dimensional Minkowski space
$\mathbb{R}^2_1$. It is well-known that any conformal
transformation on $\mathbb{R}^2$ is holomorphic or
anti-holomorphic. However, this is not the case for
$\mathbb{R}^2_1$. In this paper, the general form of conformal
transformation on $\mathbb{R}^2_1$ is obtained and differences of
conformal structures between $\mathbb{R}^2$ and $\mathbb{R}^2_1$
is discussed. To obtain this we characterize conformal
transformation by use of Lapalcian or d'Alembertian.

\section{A characterization of conformal transformation} \label{section:2}

Let $\mathbb{R}^n_\nu$ be a semi-Euclidean space with the metric
$\eta( \emph{x}, \emph{y}) = x_1 y_1 + \cdots + x_{n-\nu}
y_{n-\nu} - x_{n-\nu+1} y_{n-\nu+1} - \cdots - x_n y_n$. If we
consider $\eta$ as a $(0,2)$-tensor, then $\eta = (\eta_{ij})$ has
the component $\eta_{ij} = \epsilon_i \delta_{ij}$ where
$\epsilon_i$ is $1$ for $1 \leq i \leq n-\nu$ and $-1$ for
$n-\nu+1 \leq i \leq n$.

By conformal transformation, we mean a conformal diffeomorphism
defined on an open subset of $\mathbb{R}^n_\nu$. If $F : U \subset
\mathbb{R}^n_\nu \rightarrow V$ is a surjective conformal
transformation, then $U$ and $V$ are called conformally equivalent
and $F$ is called a conformal equivalence.. Let $M$ be a
semi-Riemannian manifold with a metric $g$ and $\nabla$ be the
unique Levi-Civita connection on $M$. Then, its Laplacian
$\triangle f$ is defined to be the divergence of gradient of $f$.

Let $F : U \subset \mathbb{R}^n_\nu \rightarrow \mathbb{R}^n_\nu$
be a $C^2$-diffeomorphism given by $(y_1, \cdots, y_n) = F(x_1,
\cdots, x_n)$, where $U$ is an open subset. Since $F$ is a
diffeomorphism, we can consider $F$ as a coordinate
transformation, and so we denote the Laplacian in terms of $x_i$'s
by $\triangle$, and the Laplacian in terms of $y_j$'s by
$\triangle^\prime$.

\begin{thm} \label{main1}
Let $U$ be an open subset of $\mathbb{R}^n_\nu$ and $F : U
\rightarrow \mathbb{R}^n_\nu$ be a $C^2$-diffeomorphism given by
$(y_1, \cdots, y_n) = F(x_1, \cdots, x_n)$. Assume that, for any
$C^2$-function $\varphi$ defined on $\mathbb{R}^n_\nu$, $\triangle
\varphi =0$ if and only if $\triangle^\prime \varphi = 0$.
Furthermore, when $\nu \neq 0, n$, we assume $\eta(\nabla y_1,
\nabla y_1)
> \eta(\nabla y_n, \nabla y_n)$. Then $F$ is a conformal
transformation.
\end{thm}
\begin{proof}
By chain rule, we have $\frac{\partial \varphi}{\partial x_i} =
\sum\limits_{j=1}^n \frac{\partial \varphi}{\partial y_j}
\frac{\partial y_j}{\partial x_i}$ and $\frac{\partial^2
\varphi}{\partial x_i^2} = \sum\limits_{j,k=1}^{n}
\frac{\partial^2 \varphi}{\partial y_j\partial y_k} \frac{\partial
y_k}{\partial x_i} \frac{\partial y_j}{\partial x_i} +
\sum\limits_{j=1}^n \frac{\partial \varphi}{\partial y_j}
\frac{\partial^2 y_j}{\partial x_i^2}$.

Therefore, we have

\begin{eqnarray*}
\triangle \varphi & = &\sum\limits_{i=1}^n \epsilon_i
\frac{\partial^2 \varphi}{\partial x_i^2} \\
& = & \sum\limits_{i=1}^n \epsilon_i \Big[ \sum\limits_{j,k=1}^{n}
\frac{\partial^2 \varphi}{\partial y_j\partial y_k} \frac{\partial
y_k}{\partial x_i} \frac{\partial y_j}{\partial x_i} +
\sum\limits_{j=1}^n \frac{\partial \varphi}{\partial y_j}
\frac{\partial^2 y_j}{\partial x_i^2} \Big].
\end{eqnarray*}
If we let $\varphi = y_j$, then since $\triangle^\prime \varphi =
0$, we have $\triangle \varphi = \triangle y_j = 0$. If we
substitute this into the above equation, we have

$$ \triangle \varphi = \sum\limits_{i=1}^n \epsilon_i \Big[ \sum\limits_{j,k=1}^{n}
\frac{\partial^2 \varphi}{\partial y_j\partial y_k} \frac{\partial
y_k}{\partial x_i} \frac{\partial y_j}{\partial x_i} \Big].$$

If we let $\varphi = y_j y_k$ with $j \neq k$, then since
$\triangle^\prime \varphi = 0$, we have $\triangle \varphi = 0$.
Therefore, if we put $\varphi = y_j y_k$ into the above equation,
we have
$$ \sum\limits_{i=1}^n \epsilon_i \Big( \frac{\partial
y_j}{\partial x_i} \frac{\partial y_k}{\partial x_i} \Big) = 0.$$

This tells us that rows of $\Big( \frac{\partial y_i}{\partial
x_j} \Big)$ are mutually orthogonal. If we consider this, we have

\begin{eqnarray*}
\triangle \varphi & = & \sum\limits_{j,k=1}^n \frac{\partial^2
\varphi}{\partial y_j}{\partial y_k} \Big( \sum\limits_{i=1}^n
\epsilon_i \frac{\partial y_j}{\partial x_i} \frac{\partial
y_k}{\partial x_i} \Big)\\
& = & \sum\limits_{i,j=1}^n \epsilon_i \frac{\partial^
\varphi}{\partial y_j^2} \Big( \frac{\partial y_j}{\partial x_i}
\Big)^2.
\end{eqnarray*}

We now consider two separate cases.

Case I : $\nu \neq 0, n$.\\
If we let $\varphi = u_k^2 + y_n^2$ for $1 \leq k \leq n-\nu$,
then we have $\triangle^\prime \varphi = 0$ and thus $\triangle
\varphi = \sum\limits_{i,j=1}^n \epsilon_i \frac{\partial^
\varphi}{\partial y_j^2} \Big( \frac{\partial y_j}{\partial x_i}
\Big)^2 = 0$.

Therefore, we have
$$ \sum\limits_{i=1}^n \epsilon_i \Big[ \Big(\frac{\partial
y_k}{\partial x_i} \Big)^2 + \Big( \frac{\partial y_n}{\partial
x_i} \Big)^2 \Big] = 0,$$ for each $1 \leq k \leq n-\nu$.

This tells us that, for $1 \leq k \leq n-\nu$, each $k$-th row has
the same length as that of the $n$-th row.

If we let $\varphi = y_1^2 + y_k^2$ for $n-\nu+1 \leq k \leq n$,
then we have $\triangle^\prime \varphi = 0$ and thus $\triangle
\varphi = \sum\limits_{i,j=1}^n \epsilon_i \frac{\partial^
\varphi}{\partial y_j^2} \Big( \frac{\partial y_j}{\partial x_i}
\Big)^2 = 0$.

Therefore, we have
$$ \sum\limits_{i=1}^n \epsilon_i \Big[ \Big(\frac{\partial
y_1}{\partial x_i} \Big)^2 + \Big( \frac{\partial y_k}{\partial
x_i} \Big)^2 \Big] = 0,$$ for each $n-\nu+1 \leq k \leq n$.

This tells us that, for $n-\nu+1 \leq k \leq n$, each $k$-th row
has the same length as that of the first row.

Since $\eta(\nabla y_1, \nabla y_1)
> \eta(\nabla y_n, \nabla y_n)$, the first $n-\nu$ rows are
spacelike and the last $\nu$ rows are timelike. Since they all
have the same length, the map $(x_1, \cdots, x_n) \mapsto (y_1,
\cdots, y_n)$ is conformal.

Case II : $\nu=0$ or $\nu = n$.\\
It suffices to consider the case $\nu = 0$.\\
If we let $\varphi = y_1^2 - y_k^2$, then by the same argument as
in the case I, we can show that the map $(x_1, \cdots, x_n)
\mapsto (y_1, \cdots, y_n)$ is conformal.

\end{proof}

When $\nu \neq 0, n$, the condition $\eta(\nabla y_1, \nabla y_1)
> \eta(\nabla y_n, \nabla y_n)$ is essential. For example, when
$2\nu=n$, the map $(x_1, \cdots, x_\nu, x_{\nu+1}, \cdots, x_n)
\mapsto (x_{\nu+1}, \cdots, x_n, x_1, \cdots, x_\nu)$ satisfies
the condition $\triangle \varphi = 0 \Leftrightarrow
\triangle^\prime \varphi = 0$ but the map is not conformal.

We now consider the converse of the above theorem. For this, we
need the following theorem, called Liouville's theorem.

\begin{thm}
Every $C^4$ conformal transformation of a region of
pseudo-Euclidean space of dimension $\geq 3$, is a composite of
isometries, dilations and inversions.
\end{thm}
\begin{proof}
This is Theorem 15.2 in Ref. \cite{DFN}.
\end{proof}

In fact, this theorem was first proved in dimensiona 3 by
Liouville in 1850 for $C^3$ transformations. In 1958, Hartman had
shown that the result holds for $C^1$ transformations(See
\cite{Hartman}), and so we now assume that the result holds for
any $C^2$ transformations.

If we consider conformal transformation defined on the whole of
$\mathbb{R}^n_\nu$, inversions can not occur, and so we have the
following.

\begin{thm}
Let $F : \mathbb{R}^n_\nu \rightarrow \mathbb{R}^n_\nu$ be a
conformal transformation with $n \geq 3$ given by $(y_1, \cdots,
y_n) = F(x_1, \cdots, x_n)$. Then we have
$$ (y_1, \cdots, y_n)^t = \alpha A (x_1, \cdots, x_n)^t +
\emph{b},$$ where the superscript $t$ means the transpose,
$\alpha$ is a real number and $A$ is an orthogonal matrix.
\end{thm}

\begin{thm}
Let $F : \mathbb{R}^n_\nu \rightarrow \mathbb{R}^n_\nu$ be a $C^2$
conformal transformation given by $(y_1, \cdots, y_n) = F(x_1,
\cdots, x_n)$ with $n \geq 3$. Then, for any $C^2$-function
$\varphi$ defined on $\mathbb{R}^n_\nu$, we have $\triangle
\varphi = 0$ if and only if $\triangle^\prime \varphi = 0$.
Furthermore, if $\nu \neq 0, n$, then we have $\eta(\nabla y_1,
\nabla y_1) > \eta(\nabla y_n, \nabla y_n)$.
\end{thm}
\begin{proof}

 By the above Theorem, we have $y_i =
\alpha \sum\limits_{j=1}^n a_{ij}x_j +b_i$ where $(a_{ij})$ is an
orthogonal matrix. By the chain rule, we have $\frac{\partial^2
\varphi}{\partial x_i^2} = \alpha \sum\limits_{j,k=1}^n
a_{ji}a_{ki} \frac{\partial^2 \varphi}{\partial y_j \partial
y_k}$. Therefore, we have
\begin{eqnarray*}
\triangle \varphi & = & \sum\limits_{i=1}^n \epsilon_i
\frac{\partial^2 \varphi}{\partial x_i^2}\\
& = & \alpha^2 \sum\limits_{i,j,k=1}^n \epsilon_i a_{ji}a_{ki}
\frac{\partial^2 \varphi}{\partial y_j \partial y_k}\\
& = & \alpha^2 \sum\limits_{j,k=1}^n \epsilon_j \delta_{jk}
\frac{\partial^2 \varphi}{\partial y_j \partial y_k}\\
& = & \alpha^2 \sum\limits_{j=1}^n \epsilon_j \frac{\partial^2
\varphi}{\partial y_j^2}\\
& = & \alpha^2 \triangle^\prime \varphi,
\end{eqnarray*}
and so we have $\triangle \varphi = 0$ if and only if
$\triangle^\prime \varphi = 0$.\\
When $\nu \neq 0, n$, since $\nabla y_1$ and $\nabla y_n$ are the
first and $n$-th rows of the orthogonal matrix $(a_{ij})$, they
are spacelike and timelike, respectively and thus we have
$\eta(\nabla
y_1, \nabla y_1) > \eta(\nabla y_n, \nabla y_n)$.\\

\end{proof}

We remark that the above theorem does not hold if the domain of
definition is a proper subset of $\mathbb{R}^n_\nu$ because of
inversions. However, when $n=2$, we can confine the domain to any
 open subsets.

\begin{thm} \label{main2}
Let $U$ be an open subset of $\mathbb{R}^2_\nu$ and $F : U
\rightarrow  \mathbb{R}^2_\nu$ be a $C^2$ conformal transformation
given by $(y_1, y_2) = F(x_1, x_2)$. Then, for any $C^2$ function
$\varphi$ defined on $U$, we have $\triangle \varphi = 0$ if and
only if $\triangle^\prime \varphi = 0$. Furthermore, if $\nu \neq
0, 2$, then we have $\eta(\nabla y_1, \nabla y_1) > \eta(\nabla
y_2, \nabla y_2)$.

\end{thm}

\begin{proof}
Let $g = F^\ast \eta$ be the pull back of $\eta$ through $F$.
Then, $(U, \eta)$ and $(U,g)$ are conformally equivalent and by
calculation, we have  $\triangle^g \varphi = 0$ if and only if
$\triangle \varphi = 0$, where $\triangle^g$ is the Laplacian with
respect to the metric $g$. Since we can consider the conformal
transformation $F$ as an isometry from $(U,g)$ onto $(F(U),
\eta)$, we have $\triangle^g = 0$ if and only if $\triangle^\prime
= 0$. Therefore, $\triangle \varphi = 0$ if and only if
$\triangle^\prime \varphi = 0$.

 When $\nu = 1$, since spacelike
vector (timelike vector, respectively) must be sent to spacelike
vector (timelike vector, respectively), we have $\eta(\nabla y_1,
\nabla y_1) > \eta(\nabla y_2, \nabla y_2)$.
\end{proof}

In conclusion, by combining Theorem \ref{main1} and Theorem
\ref{main2}, we have the following theorem.

\begin{thm} \label{final}
Let $U$ be an open subset of $\mathbb{R}^2_\nu$ and $F : U
\rightarrow \mathbb{R}^2_\nu$ be a $C^2$ diffeomorphism. Then the
necessary and sufficient condition for $F$ to be a conformal
transformation is that $\triangle \varphi = 0$ if and only if
$\triangle^\prime \varphi = 0$ for any $C^2$ function $\varphi$
and furthermore, if $\nu \neq 0, 2$, then $\eta(\nabla y_1, \nabla
y_1)
> \eta(\nabla y_2, \nabla y_2)$.
\end{thm}

\section{Conformal Transformation on $\mathbb{R}^2_\nu$} \label{section:3}

As we have seen in Liouville's Theorem, there are some rigidity in
conformal geometry on $\mathbb{R}^n_\nu$ when $n \geq 3$. However,
in the Euclidean plane $\mathbb{R}^2$, there are abundant
conformal transformations. It is a well-known fact that any
holomorphic or anti-holomorphic functions are conformal if
$f^\prime(z) \neq 0$. We can prove this by use of Theorem
\ref{final}. For this we denote $\frac{\partial^2}{\partial x^2} +
\frac{\partial^2}{\partial y^2}$ by $\triangle$ and
$\frac{\partial^2}{\partial u^2} + \frac{\partial^2}{\partial
v^2}$ by $\triangle^2$.

\begin{thm}
Let $U$ be an open subset of $\mathbb{R}^2$ and $F : U \rightarrow
\mathbb{R}^2$ be a $C^2$ conformal map given by $(u,v) = F(x,y)$.
Then, $F$, as a function of the complex number $z$, is holomorphic
or anti-holomorphic.
\end{thm}

\begin{proof}
By Theorem \ref{final}, it is sufficient to find a diffeomorphism
$(x,y) \mapsto (u,v)$ that satisfies $\triangle \varphi = 0$ if
and only if $\triangle^\prime = 0$ for any $C^2$ function
$\varphi$. By chain rule, we have
\begin{eqnarray*}
\triangle \varphi &=& \frac{\partial^2 \varphi}{\partial u^2}
\Big[ \Big(\frac{\partial u}{\partial x}\Big)^2 +
\Big(\frac{\partial u}{\partial y}\Big)^2 \Big] + \frac{\partial^2
\varphi}{\partial v^2} \Big[ \Big(\frac{\partial v}{\partial
x}\Big)^2 + \Big(\frac{\partial v}{\partial y}\Big)^2 \Big]\\
&+& 2\frac{\partial^2 \varphi}{\partial u \partial v} \Big[
\frac{\partial u}{\partial x} \frac{\partial v}{\partial x} +
\frac{\partial y}{\partial y} \frac{\partial v}{\partial y} \Big]
+ \frac{\partial \varphi}{\partial u} \Big[ \frac{\partial ^2
u}{\partial x^2} + \frac{\partial^2 u}{\partial y^2} \Big] \\
&+& \frac{\partial \varphi}{\partial v} \Big[ \frac{\partial^2
v}{\partial x^2}+\frac{\partial^2 v}{\partial y^2} \Big].
\end{eqnarray*}

If we let $\varphi = u$ and $\varphi = v$, then since
$\triangle^\prime u = \triangle^\prime v =0$, we have
$\frac{\partial^2 u}{\partial x^2} + \frac{\partial u}{\partial
y^2} = 0$ and $\frac{\partial^2 v}{\partial x^2} + \frac{\partial
v}{\partial y^2} = 0$ and thus $u$ and $v$ are analytic. \\
If we substitute $\varphi = u^2 - v^2$, we obtain
$\Big(\frac{\partial u}{\partial x}\Big)^2 + \Big(\frac{\partial
u}{\partial y}\Big)^2 = \Big(\frac{\partial v}{\partial x}\Big)^2
+ \Big(\frac{\partial v}{\partial y}\Big)^2$. By substituting
$\varphi = uv$, we obtain $\frac{\partial u}{\partial
x}\frac{\partial v}{\partial x} + \frac{\partial u}{\partial
y}\frac{\partial v}{\partial y}=0$. Multiplying both sides of
$\Big(\frac{\partial u}{\partial x}\Big)^2 + \Big(\frac{\partial
u}{\partial y}\Big)^2 = \Big(\frac{\partial v}{\partial x}\Big)^2
+ \Big(\frac{\partial v}{\partial y}\Big)^2$ by
$\Big(\frac{\partial u}{\partial y}\Big)^2$ and using
$\frac{\partial u}{\partial x}\frac{\partial v}{\partial x} +
\frac{\partial u}{\partial y}\frac{\partial v}{\partial y}=0$, we
have $\frac{\partial u}{\partial x} = \pm \frac{\partial
v}{\partial y}$ and $\frac{\partial u}{\partial y} = \mp
\frac{\partial v}{\partial x}$. Therefore, the map $(x,y) \mapsto
(u,v)$ is either holomorphic or anti-holomorphic.

\end{proof}

It is a well-known fact that real and imaginary part of
holomorphic or anti-holomorphic functions are analytic. Even
without this, we actually have shown that $u$ and $v$ are analytic
in the above proof. Therefore, though we assumed that $F$ is $C^2$
in the above theorem, we actually obtained all conformal maps on
$\mathbb{R}^2$.

We now consider conformal transformation on $\mathbb{R}^2_1$, the
two-dimensional Minkowski space with metric signature $\eta = (1,
-1)$.

Let $F : \mathbb{R}^2_1 \rightarrow \mathbb{R}^2_1$ be a $C^2$
diffeomorphism given by $(X, T) = F(x, t)$. We denote
$\frac{\partial^2}{\partial x^2}- \frac{\partial^2}{\partial t^2}$
by $\triangle$ and $\frac{\partial^2}{\partial X^2} -
\frac{\partial^2}{\partial T^2}$ by $\triangle^\prime$. Then we
have the following.

\begin{thm} \label{final-2}
Let $F : U \subset \mathbb{R}^2_1 \rightarrow \mathbb{R}^2_1$ be a
$C^2$ conformal transformation. Then, there exist two
diffeomorphisms $\chi$ and $\psi$ defined on open subsets of
$\mathbb{R}$ such that if $\chi^\prime \cdot \psi^\prime > 0 $,
then  $$ F(x,t) = \Big( \chi(x+t) + \psi(x-t), -\chi(x+t) +
\psi(x-t) \Big), $$ and if $\chi^\prime \cdot \psi^\prime < 0$,
then $$F(x,t) = \Big( \chi(x+t)-\psi(x-t), \chi(x+t) + \psi(x-t)
\Big).$$
\end{thm}

\begin{proof}
By Theorem \ref{final}, it is sufficient to find a diffeomorphism
$(x,t) \mapsto (X,T)$ that satisfies $\triangle \varphi = 0$ if
and only if $\triangle^\prime = 0$ for any $C^2$ function
$\varphi$ and $\eta(\nabla X, \nabla X)
> \eta(\nabla T, \nabla T).$

By chain rule, we have
\begin{eqnarray*} \frac{\partial^2
\varphi}{\partial x^2} - \frac{\partial^2 \varphi}{\partial t^2}
&=& \frac{\partial^2 \varphi}{\partial X^2} \Big[ \Big(
\frac{\partial X}{\partial x} \Big)^2 - \Big(\frac{\partial
X}{\partial t} \Big)^2 \Big] + \frac{\partial^2 \varphi}{\partial
T^2} \Big[ \Big( \frac{\partial T}{\partial x} \Big)^2 - \Big(
\frac{\partial T}{\partial t} \Big)^2 \Big]\\
& & \,\,\,\, + 2 \frac{\partial^2 \varphi}{\partial X
\partial T} \Big[ \frac{\partial X}{\partial x} \frac{\partial
T}{\partial x} - \frac{\partial X}{\partial t} \frac{\partial
T}{\partial t} \Big] + \frac{\partial \varphi}{\partial X} \Big[
\frac{\partial^2 X}{\partial x^2} - \frac{\partial^2 X}{\partial
t^2} \Big]\\ & & \,\,\, + \frac{\partial \varphi}{\partial T}
\Big[ \frac{\partial^2 T}{\partial x^2} - \frac{\partial^2
T}{\partial t^2} \Big] \,\,\, \cdots (*).
\end{eqnarray*}

By substituting $\varphi = X$ and $\varphi = T$ into the above
equation, we have $\frac{\partial^2 X}{\partial x^2} -
\frac{\partial^2 X}{\partial t^2} = 0$ and $\frac{\partial^2
T}{\partial x^2} - \frac{\partial^2 T}{\partial t^2} = 0$. In
other words, $X$ and $T$ are solutions of wave equations.
Therefore, there are $f$, $g$, $k$ and $h$ such that
$$ \,\, \,\,\,\,\,\, X = f(x+t) + g(x-t), and $$
$$ T = k(x+t) + h(x-t).$$

To find the explicit form, we substitute this into $\eta(\nabla X,
\nabla X) > \eta(\nabla T, \nabla T).$ Then, we obtain $$f^\prime
g^\prime > k^\prime h^\prime \,\,\,\,\,\, \cdots (1).$$

By substituting $\varphi = XT$ into $(*)$, we get $\frac{\partial
X}{\partial x} \frac{\partial T}{\partial x} - \frac{\partial
X}{\partial t} \frac{\partial T}{\partial t} = 0$, and then we
have
$$ f^\prime h^\prime + g^\prime k^\prime = 0 \,\,\,\,\,\, \cdots
(2).$$

By substituting $\varphi = X^2 + T^2$ into $(*)$, we get $\Big(
\frac{\partial X}{\partial x} \Big)^2 - \Big( \frac{\partial
X}{\partial t} \Big)^2 = - \Big[ \Big( \frac{\partial T}{\partial
x} \Big)^2 - \Big( \frac{\partial T}{\partial t} \Big)^2 \Big]$
and then, we have
$$ f^\prime g^\prime + k^\prime h^\prime = 0 \,\,\,\,\,\, \cdots (3).$$

By combining (1) and (3), we can conclude that
$$ f^\prime g^\prime >0 \,\,\, \mbox{and} \,\,\, k^\prime h^\prime
< 0 \,\,\,\,\,\, \cdots(4).$$

From this, we can see that each of $f$, $g$, $k$ and $h$ is a
diffeomorphism defined on an open subset of $\mathbb{R}$.

From equation $(2)$, we have $\frac{f^\prime}{g^\prime} = -
\frac{k^\prime}{h^\prime} = \alpha$ for some positive function
$\alpha$. If we substitute this into equation $(3)$, we obtain
$g^\prime = \pm h^\prime$ and $f^\prime = \mp k^\prime$.

Assume that $g^\prime = h^\prime$ and $f^\prime = -k^\prime$.\\
 In this case, we have $g = h + c$ and $k = -f + d$ for some
 constants $c$ and $d$ and finally, we have
 \begin{eqnarray*}
 X &=& f(x+t) + h(x-t) +c, \,\,\,\, \mbox{and} \\
  T &=& -f(x+t) + h(x-t) + d.
  \end{eqnarray*}
If we let $\chi(u) = f(u) + \frac{1}{2}(c-d)$ and $\psi(v) = h(v)
+ \frac{1}{2}(c+d)$, then we have
$$ X = \chi(x+t) + \psi(x-t) \,\,\, \mbox{and} \,\,\, T =
-\chi(x+t) + \psi(x-t),$$ where $\chi$ and $\psi$ are
diffeomorphisms. Note that since $g^\prime = h^\prime$ and
$f^\prime = -k^\prime$, by $(4)$, either both $\chi$ and $\psi$
are increasing or both are decreasing.

We now assume that $g^\prime = -h^\prime$ and $f^\prime =
k^\prime$.

Then, we have $g = -h + c$ and $f = k + d$ for some constants $c$
and $d$ and thus, we have
 \begin{eqnarray*}
 X &=& f(x+t) - h(x-t) + c, \,\,\,\, \mbox{and} \\
  T &=& f(x+t) + h(x-t) - d.
  \end{eqnarray*}
If we let $\chi(u) = f(u) + \frac{1}{2}(c-d)$ and $\psi(v) = h(v)
- \frac{1}{2}(c+d)$, then we have
$$ X = \chi(x+t) - \psi(x-t) \,\,\, \mbox{and} \,\,\, T =
\chi(x+t) + \psi(x-t).$$ Note also that, in this case, either
$\chi$ is increasing and $\psi$ is decreasing or $\chi$ is
decreasing and $\psi$ is increasing.

\end{proof}

In $\mathbb{R}^2_1$ with coordinate $(x,t)$, if we introduce a
null coordinate $u = x + t$ and $v = x - t$, the above form can be
simplified.

\begin{thm} \label{null}
Let $F : U \subset \mathbb{R}^2_1 \rightarrow \mathbb{R}^2_1$ be a
$C^2$ conformal transformation given by $(U,V) = F(u,v)$ in terms
of null coordinates. Then there are two diffeomorphisms $\chi$ and
$\psi$ defined on open subsets on $\mathbb{R}$, which are either
both increasing or both decreasing, such that $F$ is given by
either $F(u,v) = ( \psi(u), \chi(v) )$ or $F(u,v) = ( \psi(v),
\chi(u))$
\end{thm}
\begin{proof}
In the previous theorem, if $\chi^\prime \cdot \psi^\prime > 0$,
then it is easy to see that $F$ is given by $F(u,v) = \big(
2\psi(v), 2\chi(u) \big)$. If we replace $2\psi$ and $2\chi$ by
$\psi$ and $\chi$, we have $F(u,v) = \big(\psi(v), \chi(u) \big)$.

Likewise, if $\chi^\prime \cdot \psi^\prime < 0$, then $F$ is
given by $F(u,v) = \big( 2\chi(u), -2\psi(v) \big)$. If we replace
$2\chi$ and $-2\psi$ by $\chi$ and $\psi$, then they have the same
increasing pattern and $F$ is given by $F(u,v) = \big( \chi(u),
\psi(v) \big)$.
\end{proof}

The domain of definition of $\chi$ and $\psi$ depend on the
geometry of $U$ and the geometry of $F(U)$ depends on the range of
$\chi$ and $\psi$. If we take $\chi$ and $\psi$ to be  bijective
diffeomorphisms defined on the whole of $\mathbb{R}$, then we can
get the general form of conformal equivalences from
$\mathbb{R}^2_1$ onto $\mathbb{R}^2_1$ and so we have the
following.

\begin{corollary}
Let $F : \mathbb{R}^2_1 \rightarrow \mathbb{R}^2_1$ be a
surjective diffeomorphism. The necessary and sufficient condition
for $F$ to be conformal is that there exist surjective
diffeomorphisms $\psi$ and $\chi$ which are either both increasing
or both decreasing such that $F(u,v) = \big( \psi(u), \chi(v)
\big)$ or $F(u,v) = \big( \psi(v), \chi(u) \big)$.
\end{corollary}

\section{Counterparts to the Riemann mapping theorem}

As we have seen in previous sections, Euclidean 2-space and
Minkowski 2-space have abundant conformal transformations compared
to higher dimensional spaces. Concerned with conformal structure
of $\mathbb{R}^2$, we have the famous theorem, called the Riemann
mapping theorem as the following.

\begin{thm}
Let $U$ be an open, simply-connected, proper subset of
$\mathbb{R}^2$. Then, there exits a conformal equivalence from $U$
onto the open unit disk $D$.
\end{thm}

From the Riemann mapping theorem, we can see that any two
simply-connected, proper subset of $\mathbb{R}^2$ are conformally
equivalent to each other.

 In $\mathbb{R}^2_1$, this theorem does not hold and to see a
difference between $\mathbb{R}^2$ and $\mathbb{R}^2_1$ clearly, we
note that in the Riemann mapping theorem, the condition for $U$ to
be proper is essential since entire function can not be bounded.
In other words, $\mathbb{R}^2$ is not conformally equivalent to
any bounded subset of $\mathbb{R}^2$. However, $\mathbb{R}^2_1$ is
conformally equivalent to a bounded subset $D_M = \{(x,t) | |x| +
|t| < 1 \}$ as the following theorem shows.

\begin{thm} \label{confom-1}
There exists a conformal equivalence from $\mathbb{R}^2_1$ onto
$D_M$.
\end{thm}
\begin{proof}
Let $\chi(x) = \psi(x) = \frac{2}{\pi} \tan^{-1}x$. Then, by
Theorem \ref{null}, the map $F(u,v) = \big( \chi(u), \psi(v)
\big)$ defined in terms of null coordinates, is the desired
conformal equivalence.

\end{proof}

From this theorem, we can state the following theorem similar to
Riemann mapping theorem.

\begin{thm}
Any open subset of $\mathbb{R}^2_1$ is conformally equivalent to a
subset of $D_M$.
\end{thm}
\begin{proof}
The image of any open subset under the conformal equivalence
constructed above theorem lies in $D_M$.
\end{proof}

In general, conformal transformation is either a causal
isomorphism or an anti-causal isomorphism and thus, if $D = \{
(x,t) | x^2 + t^2 <1 \}$ and $D_M$ are conformally equivalent to
each other, then they have the same type of causal structure.
However, $D_M$ is globally hyperbolic, but $D$ is not. Therefore,
in $\mathbb{R}^2_1$, $D$ and $D_M$ are not conformally equivalent.

We now analyze the structure of the group of conformal
transformations on a bounded open subset $U$ of $\mathbb{R}^2_1$.
Since $U$ is bounded, we can find $\alpha_i$ and $\beta_i$ for $i
= 1,2 $ such that $\alpha_1 < u < \alpha_2$ and $\beta_1 < v <
\beta_2$ for all $(u,v)$ in $U$, where $u = x+t$ and $v = x-t$.
Let $a_1 = \sup \{ \alpha_1 \}$, $a_2 = \inf \{ \alpha_2 \}$, $b_1
= \sup \{ \beta_1 \}$ and $b_2 = \inf \{ \beta_2 \}$. Then, we
have $U \subset D_U = \{(u,v) \,\, | \,\, a_1 < u < a_2, \,\, b_1
< v < b_2 \}$.

\begin{lemma}
For any bounded open subset $U$ of $\mathbb{R}^2_1$, $D_M$ and
$D_U$ are conformally equivalent.
\end{lemma}
\begin{proof}
Let $\psi(u) = \frac{1}{2} \big[ (a_2 - a_1)u + (a_1+a_1) \big]$
and $\chi(v) = \frac{1}{2} \big[ (b_2-b_1)u + (b_1+b_1) \big]$.
Then, $F(u,v) = \big( \psi(u), \chi(v) \big)$ is a conformal
equivalence from $D_M$ onto $D_U$ by Theorem \ref{null}.
\end{proof}

In the above proof, we can see that any two rectangles whose sides
are parallel to the axes of null coordinates, are conformally
equivalent to each other, which is not the case in $\mathbb{R}^2$.
In fact, it is known that there does not exist a conformal map
from a square onto a non-square rectangle which maps the vertices
to the vertices.(See pp. 14-15 in \cite{Lehto}.)

 We know that $\mathbb{R}^2_1$ is conformally equivalent to $D_M$
 and from the above lemma, we can conclude that $D_U$ is
 conformally equivalent to $\mathbb{R}^2_1$. Therefore, we have
 the following.

\begin{corollary}
 The group of conformal transformations from $D_U$ onto $D_U$ is isomorphic to
 that of $\mathbb{R}^2_1$
 \end{corollary}

For a given bounded, open subset $U$ of $\mathbb{R}^2_1$, we can
find $D_U$ and, as we have remarked, the domains of definition of
$\psi$ and $\chi$ depend on the geometry of $U$. To be precise, if
$F : (u,v) \mapsto \big( \psi(u), \chi(v) \big)$ is a conformal
equivalence from $U$ onto $U$ itself then, since $\psi$ and $\chi$
has been defined on $\{ u \, | \, (u,v) \in U \,\,\, \mbox{for
some} \,\,\, v \}$ and $\{ v \, | \, (u,v) \in U \,\,\, \mbox{for
some} \,\,\, u \}$, $F$ can be uniquely extended to $D_U$. Since
$\mathbb{R}^2_1$ and $D_U$ are conformally equivalent, we have the
following.

\begin{thm}
For any bounded, open subset $U$ of $\mathbb{R}^2_1$, the group of
conformal equivalences on $U$ is isomorphic to a subgroup of the
group of conformal equivalences on $\mathbb{R}^2_1$.
\end{thm}

Obviously, the above theorem does not hold in $\mathbb{R}^2$ and
finally, we have the following.

\begin{thm}
Let $U$ and $V$ be bounded, open subsets of $\mathbb{R}^2_1$. The
necessary and sufficient conditions for $U$ and $V$ to be
conformally equivalent is that there exist diffeomorphisms $\psi$
and $\chi$ defined on bounded subsets of $\mathbb{R}$ such that
$\psi$ and $\chi$ have the same increasing patterns and either
$F(u,v) = \big( \psi(u), \chi(v) \big)$ or $F(u,v) = \big(
\psi(v), \chi(u) \big)$ induces a bijection from $U$ onto $V$.
\end{thm}

\section{Acknowledgement}

This research was supported by Basic Science Research Program
through the National Research Foundation of Korea(NRF) funded by
the Ministry of Education, Science and Technology(2013052685).

\end{document}